\documentclass[12pt]{article}
\usepackage{amsmath}
\usepackage{amsfonts}
\usepackage{amssymb}
\usepackage{amsthm}
\usepackage{stackrel}

\usepackage{color}

\usepackage{figsize}

\usepackage{graphics}

\newtheorem{thm}{Theorem}[section]

\newcommand{\go}[1]{\mathfrak{#1}}

\newcommand{\R}{{\rm I}\kern-0.18em{\rm R}}
\newcommand{\1}{{\rm 1}\kern-0.25em{\rm I}}
\newcommand{\E}{{\rm I}\kern-0.18em{\rm E}}
\newcommand{\p}{{\rm I}\kern-0.18em{\rm P}}

\makeatletter

\def\@fnsymbol#1{\ensuremath{\ifcase#1\or a\or b\or c\or d\or \e\or f\or *\dagger 	\or \ddagger\ddagger \else\@ctrerr\fi}}
\title{Some More Results on Characterization of the Exponential Distribution}
\author{Lev B. Klebanov\footnote{Department of Probability and Mathematical Statistics, Charles University, Prague,
Czech Republic. e-mail: lev.klebanov@mff.cuni.cz} and Zeev E. Vol'kovich\footnote{Software Engineering Department, ORT Braude College, P.O. Box: 78, Karmiel, 21982, Israel}}
\date{}
\begin{document}

\maketitle

\begin{flushright}
Dedicated to the blessed memory \\ of professor Abram Aronovich Zinger
\end{flushright}

\begin{abstract}
There are given characterizations of the exponential distribution by the properties of independence of linear forms with random coefficients. Related results based on the constancy of regression of one statistic on a linear form are obtained.  
\end{abstract}

\section{Introduction}\label{s1} 
\setcounter{equation}{0} 

On October 11, 2019 passed away well-known specialist in Probability and Statistics professor Abram Aronovich Zinger. One of the main fields of his scientific interests was the theory of characterizations of probability distributions. He published a number of outstanding results in this field (see \cite{Z51},\cite{ZL57},\cite{Z69},\cite{ZL70}). Many of these results are connected with characterization of the normal law by the independence and/or identical distribution property of suitable statistics. In cooperation with professor Yuri Vladimirovich Linnik he was the first who provide such characterization using linear forms with random coefficients \cite{ZL70}. However, in our discussions with professor Zinger, he expressed an opinion the slightly different properties may characterize other classes of distributions. Now we are talking on independence properties of linear forms with random coefficients because characterizations of different probability distribution classes by independence of non-linear statistics are known for a very long time (see, for example, \cite{ZL57}, \cite{Z58}, \cite{ZKM}). Some characterizations by of identical distribution property and a constance of regression are known as well \cite{Z69},\cite{ZK90}. In two latest publications (\cite{K19a}, \cite{K19b}) their author tried to show that professor Zinger opinion on possibility to use the independence of linear forms with random coefficients for characterization non-normal distribution is true. Namely, such properties are suitable for characterization of two-point and hyperbolic secant distributions. The aim of this paper is to show that the exponential distribution may be characterized in the similar way.

\section{Exponential distribution and linear form with random coefficients}\label{s2} 
\setcounter{equation}{0} 

\begin{thm}\label{th1}
Let $\varepsilon_p$ be a random variable taking values $1$ with probability $p \in (0,1)$ and $0$ with probability $1-p$. Suppose that $X,Y$ are independent identically distributed (i.i.d.) random variables positive almost surely (a.s.) and independent with $\varepsilon_p$.
Consider linear forms 
\begin{equation}\label{eq1}
L_1=(1-p)a X+\varepsilon_p a Y \quad \text{and} \quad L_2= pb X+(1-\varepsilon_p) b Y,
\end{equation}
where $a,b$ are positive constants. Then the forms $L_1$ and $L_2$ are independent if and only if $X$ and $Y$ have exponential distribution.
\end{thm}
\begin{proof}
Laplace transform of the random vector $(L_1,L_2)$ has the following form
\begin{equation}\label{eq2}
\E \exp\{-sL_1 - tL_2\} = f(a(1-p)s+bpt)\Bigl(f(as)p +f(bt)(1-p)\Bigr),
\end{equation}
where $f$ is Laplace transform of the random variable $X$. Random forms $L_1$ and $L_2$ are independent if and only if 
\[ \E \exp\{-sL_1 - tL_2\} = \E \exp\{-sL_1\} \E \exp\{-tL_2\}, \] 
what is equivalent to
\[f(a(1-p)s+bpt)\Bigl(f(as)p +f(bt)(1-p)\Bigr)= \]
\begin{equation}\label{eq3}
= f(a(1-p)s)\Bigl(f(as)p +(1-p)\Bigr)f(bpt)\Bigl(p +f(bt)(1-p)\Bigr).
\end{equation}
Change $as$ and $bt$ by new variables which we denote $s$ and $t$ again. Instead of (\ref{eq3}) we obtain
\[f((1-p)s+pt)\Bigl(f(s)p +f(b)(1-p)\Bigr) = \]
\begin{equation}\label{eq4}
= f((1-p)s)\Bigl(f(s)p +(1-p)\Bigr)f(pt)\Bigl(p +f(t)(1-p)\Bigr).
\end{equation}
By substituting $f(s) = 1/(1+\lambda s)$ ($\lambda >0$) into (\ref{eq4}) we obtain that the Laplace transform of exponential distribution satisfies this equation that is $L_1$ and $L_2$ are independent for exponentially distributed $X$ and $Y$.

Let us show inverse statement: if $L_1$ and $L_2$ are independent then $X$ and $Y$ have exponential distribution. To this aim put $t=s$ into (\ref{eq4}). We obtain
\begin{equation}\label{eq5}
f^2(s)= f(ps)f((1-p)s)\Bigl( p(1-p)f^2(s) +\bigl(p^2+(1-p)^2\bigr)f(s) +p(1-p)\Bigr).
\end{equation}
Now we would like to show the equation (\ref{eq5}) has no other solutions than $\varphi(s;\lambda) = 1/(1+\lambda s)$ ($\lambda >0$) which are Laplace transform of a probability distribution.
Suitable for this aim is the method of intensively monotone operators developed in \cite{KKM}. However, the equation (\ref{eq5}) has not too convenient form to apply any of the theorems from \cite{KKM} and, therefore, we apply the method itself. It is clear that:
\begin{enumerate}
\item $\varphi(s;\lambda)$ satisfies (\ref{eq5}) for any $\lambda >0$;
\item $\varphi(s;\lambda)$ is analytic in $s$ in some neighborhood of the point $s=0$;
\item for any positive $s_o$ there is $\lambda_o$ such that $f(s_o) = \varphi (s_o,\lambda_o)$, where $f$ is a fixed solution of (\ref{eq5}). 
\end{enumerate}
Let us consider the difference between $f(s)$ and $\varphi(s,\lambda_o)$. From 3. it follows that $\varphi(s_o;\lambda_o) = f(s_o)$. Define the set $S=\{ s: 0<s \leq s_o, f(s)=\varphi(s,\lambda_o)\}$ and denote $s^*=\inf S$. Because both $f(s)$ and $\varphi(s,\lambda_o)$ are continuous we have $f(s^*)=\varphi(s^*,\lambda_o)$. Show that the case $s^*>0$ is impossible. Really, in this case we would have
\[ f^2(s^*)= f(ps^*)f((1-p)s^*)\Bigl( p(1-p)f^2(s^*) +\bigl(p^2+(1-p)^2\bigr)f(s^*) +p(1-p)\Bigr)\]
and
\[ \varphi^2(s^*)= \varphi(ps^*,\lambda_o)\varphi((1-p)s^*,\lambda_o)\Bigl( p(1-p)\varphi^2(s^*,\lambda_o)+ \] \[+\bigl(p^2+(1-p)^2\bigr)\varphi(s^*,\lambda_o) +p(1-p)\Bigr).\]
We have $f(s^*)=\varphi(s^*,\lambda_o)$ and two previous relations give us
\begin{equation}\label{eq6}
 f(ps^*)f((1-p)s^*) = \varphi(ps^*,\lambda_o)\varphi((1-p)s^*,\lambda_o).
\end{equation}
However, $f(s) \neq \varphi(s,\lambda_o)$ for all $s\in (0,s^*)$. Therefore, either
$f(ps^*)f((1-p)s^*) > \varphi(ps,\lambda_o)\varphi((1-p)s^*,\lambda_o)$ or $f(ps^*)f((1-p)s^*) < \varphi(ps^*,\lambda_o)\varphi((1-p)s^*,\lambda_o)$ in contradiction with (\ref{eq6}). This leads us to the fact that $s^*=0$. The latest is possible only if there exists a sequence of pints $\{s_j,\; j=1,2, \ldots \}$ such that $s_j \to 0$ as $j \to \infty$ and $f(s_j)=\varphi(s_j,\lambda_o)$. It shows that $f(s)=\varphi(s,\lambda_o)$ for all $s \geq 0$ (see  Example 1.3.2 from \cite{KKM}). 
\end{proof}

{\it Let us note that under conditions of Theorem \ref{th1} we have
\begin{equation}\label{eq7}
((1-p)X+\varepsilon_p Y, pX+(1-\varepsilon_p)Y ) \stackrel{d}{=}(X,Y)
\end{equation}
if and only if $X$ has an exponential distribution}. Really, from Theorem \ref{th1} it follows that $X$ has exponential distribution. It is easy to verify that $X \stackrel{d}{=}(1-p)X +\varepsilon_p Y$ for independent exponentially distributed $X$ and $Y$.

The relation (\ref{eq7}) leads us a number of questions. Let us mention some of them:
\begin{enumerate}
\item[1)] Is the property $aX+bY \stackrel{d}{=}\bigl((1-p)a+pb\bigr)X +\bigl(a\varepsilon_p+b(1-\varepsilon_p)\bigr)Y$ characteristic for an exponential distribution?
\item[2)] Let $X_1, \ldots ,X_n, \ldots$ be a sequence of independent random variables distributed identical with $X$. Suppose that $Y,\; \varepsilon_p$ are from Theorem \ref{th1} and $\{\nu_q, \; q \in (0,1)\}$ has geometric distribution with the parameter $q$ independent on the sequence of $X_1, \ldots ,X_n, \ldots$. When $(1-p)X+\varepsilon_p Y \stackrel{d}{=} q \sum_{j=1}^{\nu_q}X_j$?
\item[3)] Does the relation
\[\E\{ pX+(1-\varepsilon_p)Y | (1-p)X+\varepsilon_p Y\} =\text{const} \]
characterize an exponential distribution?
\end{enumerate}
\vspace{-0.2cm} Below we shall try to answer these questions for some particular cases.

\vspace{0.1cm}
{\it Let us start with the question 1)}. Below we give two results in connection with this question.

\begin{thm}\label{th2}
Let $\varepsilon_p$ be a random variable taking values $1$ with probability $p \in (0,1)$ and $0$ with that of $1-p$. Suppose that $X,Y$ are independent identically distributed (i.i.d.) random variables a.s. positive and independent with $\varepsilon_p$. Linear forms $X$ and $(1-p)X+\varepsilon_p Y$ are identically distributed
\begin{equation}\label{eq8}
X \stackrel{d}{=}(1-p)X +\varepsilon_p Y
\end{equation}
if and only if $X$ has an exponential distribution.
\end{thm}
\begin{proof}
The forms are identically distributed if and only if their Laplace transforms satisfy the equation
\[ f(t)=f((1-p)t)\bigl(p f(t) +(1-p) \bigr)\]
or 
\[ f(t)= \frac{1-p}{1-p f((1-p)t)}. \]
Now the result follows from \cite{KMM, KKRT}.
\end{proof}

\begin{thm}\label{th3} Let $\varepsilon_p$ be a random variable taking values $1$ with probability $p \in (0,1)$ and $0$ with that of $1-p$. Suppose that $X,Y$ are independent identically distributed (i.i.d.) random variables a.s. positive and independent with $\varepsilon_p$. Let $0<a<b<1$. Define $V=(b-a)/\bigl(pb+(1-p)a\bigr)$ and $k=[\log(1/p)/\log(V)]+2$, where square brackets are used for integer part of the number in them. Suppose that $X$ has finite moment of order $k$. Linear forms
\begin{equation}\label{eq9}
aX+bY \stackrel{d}{=}\bigl((1-p)a+pb\bigr)X +\bigl(a\varepsilon_p+b(1-\varepsilon_p)\bigr)Y
\end{equation}
are identically distributed if and only if $X$ has an exponential distribution.
\end{thm}
\begin{proof} In terms of Laplace transformation the relation (\ref{eq9}) takes form
\begin{equation}\label{eq10}
f(as)f(bs)=f(cs)\Big(p f(as) + (1-p) f(bs)\Bigr), \quad c=(1-p)a+pb.
\end{equation}
It is easy to verify that Laplace transform of exponential distribution $g(s)=1/(1+\lambda s)$ is a solution of (\ref{eq10}) for arbitrary $\lambda>0$.
Therefore, (\ref{eq9}) holds for exponential distribution with arbitrary scale parameter. Introduce new functions $\varphi(s) =1/f(s)$ and $\psi(s)=1/g(s)$. It is clear that $\varphi(0) = \psi(0)=1$ and 
\begin{equation}\label{eq11}
\varphi(s)=p\varphi(Bs)+(1-p)\varphi(As),
\end{equation}
where $A=a/c<1$ and $B=b/c>1$. Obviously, the function $\psi$ satisfies equation (\ref{eq11}) as well. Note that (\ref{eq11}) is not Cauchy equation because $a,b$ and $c$ are fixed numbers. The function $f(s)$ is Laplace transform of a probability distribution which is not degenerate at zero. Therefore, $\varphi(s)$ is greater or equal to 1 for all positive values of $s$ and tends monotonically to infinity as $s \to \infty$. Because $X$ has moments up to order $k$ the functions $f(s)$ and $\varphi (s)$ are at least $k$ times differentiable for $s \in [0,\infty)$. It is easy to see that $\log(1/p)/\log(V) >1$ and, therefore, $k \geq 3$. It is also clear that 
\begin{equation}\label{eq12}
pB^j+(1-p)A^j \neq 1
\end{equation}
for $j=2,3,\ldots$ while the left hand side of (\ref{eq11}) coincide with $1$ for $j=0,1$. Therefore, the derivative $\varphi^{(j)}(0)$ may be arbitrary for $j=0,1$ and equals to zero for $j=2, \ldots ,k$ (to obtain this it is sufficient take $j$-th derivative from both sides of (\ref{eq11}) and setting $s=0$). 

Because the function $f(s)$ is Laplace transform of a probability distribution which is not degenerate at zero then $\varphi (0)=1$, $\varphi^{\prime}(0)>0$. From (\ref{eq11}) we have
\[ \varphi(B s) =\frac{1}{p}\Bigl( \varphi(s)-(1-p)\varphi(A s)\Bigr) <\frac{1}{p}\varphi(s). \]
Therefore,
\[ \varphi (B^m s) < \frac{1}{p^m}\varphi (s), \quad m=1,2, \ldots , \]
and, consequently, 
\[ \varphi (s) < C s^{\gamma} \]
for sufficiently large values of $s$. Here $C>0$ is a constant and $\gamma = \log(1/p)/\log(B)$. The difference $\xi (s)=\varphi (s)- \psi (s)$ satisfies equation (\ref{eq11}) and conditions:
\begin{enumerate}
\item[$a)$] $\xi^{(j)}(0)=0$ for $j=0,1, \ldots k$;
\item[$b)$] $|\xi (s)| < C s^{\gamma}$ for sufficiently large $s$.
\end{enumerate}

Introduce a space $\go F$ of real continuous functions $\zeta (s)$ on $[0,\infty)$ for which the integral $\int_{0}^{\infty}|\zeta(s)|/s^{k+1} ds$ converges. According to properties $a)$ and $b)$ we have $\xi \in {\go F}$. Define a distance $d$ on $\go F$ as 
 \[ d(\zeta_1,\zeta_2) =\int_{0}^{\infty}\bigl|\zeta_1(s) - \zeta_2(s)\bigr| \frac{ds}{s^{k+1}}. \]
It is clear that $({\go F},d)$ is a complete metric space. Introduce the following operator
\[ \mathcal{A}(\zeta) =\frac{1}{p}\Bigl( \zeta (s/B )-(1-p)\zeta (sA/B)\Bigr) \]
from ${\go F}$ to ${\go F}$. For $\zeta= \zeta_1 - \zeta_2$, where $\zeta_1, \zeta_2 \in {\go F}$, we have 
\[ d(\mathcal{A}(\zeta_1), \mathcal{A}(\zeta_2)) =\int_{0}^{\infty} \Bigl|\frac{1}{p}\Bigl( \zeta (s/B )-(1-p)\zeta (sA/B)\Bigr)\Bigr|\frac{d s}{s^{k+1}} \leq  \]
\[ \leq \int_{0}^{\infty} \Bigl|\frac{1}{p}\zeta (s/B )\Bigr|\frac{d s}{s^{k+1}} + \int_{0}^{\infty}\Bigl| \frac{1}{p}(1-p)\zeta (sA/B)\Bigr|\frac{d s}{s^{k+1}} \leq \]
\[ \leq \frac{1+(1-p)A^{k}}{p B^{k}} d(\zeta_1,\zeta_2) = \rho d(\zeta_1,\zeta_2),\]
where 
\[ \rho = \frac{1+(1-p)A^{k}}{p B^{k}} <1.\]
Now we see that $\mathcal{A}$ is contraction operator. It has only one fixed point in the space ${\go F},d$. Obviously, this point is $\zeta(s) = 0$ for all $s \geq 0$. In other words, $\xi (s) =0$ for all $s \geq 0$ and $\varphi = \psi$.
\end{proof}
Let us mention that Theorems \ref{th2} and \ref{th3} give particular answers for the question 1) above.  

\vspace{0.2cm}
{\it Went now to the question 2)}. Namely, let $X_1, \ldots ,X_n, \ldots$ be a sequence of independent random variables distributed identical with $X$. Suppose that $Y,\; \varepsilon_p$ are from Theorem \ref{th1} and $\{\nu_q, \; q \in (0,1)\}$ has geometric distribution with the parameter $q$ independent on the sequence of $X_1, \ldots ,X_n, \ldots$. When $(1-p)X+\varepsilon_p Y \stackrel{d}{=} q \sum_{j=1}^{\nu_q}X_j$? In terms of Laplace transform we have solve the following equation
\begin{equation}\label{eq13}
f((1-p)s)\Bigl((1-p)+p f(s)\Bigr)\Bigl(1-(1-q)f(qs) \Bigr)=q f(qs).
\end{equation}
The case $q=1-p$ appears to be very simple.
\begin{thm}\label{th4}
Under conditions above, if $q=1-p$ equation (\ref{eq13}) hold if and only if $f(s)$ is Laplace transform of an exponential distribution.
\end{thm}
\begin{proof} For $q=1-p$ let us make the change of function. Namely, set $1/f(s)=1+\xi(s)$. After simple transformations we come to 
\[ \xi (s) = \frac{1}{1-p}\xi ((1-p)s). \]
The statement follows now from \cite{KKM} similarly to the proof of Theorem \ref{th2}.
\end{proof}

\begin{thm}\label{th5}
Under conditions above suppose the distribution of $X$ has the moments of all orders. Let $q^{k-1}\neq p+(1-p)^k$ for all $k=2,3, \ldots$. Equation (\ref{eq13}) holds if and only if $f(s)$ is Laplace transform of an exponential distribution.
\end{thm}
\begin{proof} Let us rewrite (\ref{eq13}) in the form
\begin{equation}\label{eq14}
f((1-p)s)\Bigl((1-p)+p f(s)\Bigr)\Bigl(1-(1-q)f(qs) \Bigr)-q f(qs)=0.
\end{equation} 
Because $X$ has moments of all orders its Laplace transform is infinite differentiable for all values of $s \geq 0$. Let us differentiate $k$ times the both sides of (\ref{eq14}) with respect to $s$ and put $s=0$. The coefficient at
$f^{k}(0)$ is
\begin{equation}\label{eq15}
(1-p)^kq+qp-(1-q)q^k-q^{k+1} = q\bigl((1-p)^k+p-q^{k-1}\bigr) \neq 0
\end{equation}
for $k=2,3, \ldots$. For $k=1$ coefficient (\ref{eq15}) is zero. This means the derivatives of $f$ of order $k>1$ calculated  at zero are uniquely defined by the value of first derivative $f^{\prime}(0)<0$. However, Laplace transform of the Exponential distribution satisfies (\ref{eq14}), and its derivative at zero may be taking as arbitrary negative number.   
\end{proof}

\vspace{0.2cm}{\it Let us now went to the question 3)}. 
\begin{thm}\label{th6}
Let $X,Y$ are i.i.d. positive random variables possesing finite moments of all orders. Suppose that $\varepsilon_p$ is a Bernoulli random variate independent with $X,Y$ and takes value $1$ with probability $p$ and 0 with that of $1-p$, $0<p<1$. The relation
\begin{equation}\label{eq16}
\E\{ pX+(1-\varepsilon_p)Y | (1-p)X+\varepsilon_p Y\} =\text{const} 
\end{equation} 
holds if and only if $X$ has an exponential distribution.
\end{thm}
\begin{proof}
Equation (\ref{eq16}) may be written in terms of Laplace transform as
\begin{equation}\label{eq17}
-p f^{\prime}((1-p)t)\Bigl(pf(t) + (1-p)\Bigr) = \lambda p f((1-p)t)f(t),
\end{equation}
where $\lambda =\E X >0$. After changing function $f$ to $\varphi = 1/f$ we obtain
\begin{equation}\label{eq18}
\varphi^{\prime}\left((1-p)t \right)\Bigl( p+(1-p)\varphi (t) \Bigr) = \lambda \varphi \left((1-p)t \right). 
\end{equation}
Putting here $t=0$ and taking into account $\varphi (0) =1$ we obtain $\varphi^{\prime}(0)=\lambda$ (what is obvious).
However, differentiating both part of (\ref{eq18}) with respect to $t$ we obtain 
\begin{equation}\label{eq19}
\varphi^{\prime \prime}\bigl((1-p)t\bigr) (1-p) \Bigl( p+(1-p)\varphi(t)\Bigr) +(1-p)\varphi^{\prime}\bigl((1-p)t\bigr) \varphi^{\prime}(t) = \lambda (1-p) \varphi^{\prime}\bigl((1-p)t\bigr)
\end{equation}
Putting here $t=0$ we obtain $\varphi^{\prime \prime}(0) =0$. The induction shows that $\varphi^{(m)}(0) = 0$ for all $m=2,3, \ldots$. It implies $\varphi(t)$ is a linear polynomial in $t$ and, therefore, $f(t)$ is Laplace transform of an exponential distribution.
\end{proof}

\section{Few words in conclusion}\label{sec3} 
\setcounter{equation}{0} 

Of course, there are many other problems connected to characterizations of exponential distribution by the properties of linear forms with random coefficients. Let us mention some questions on identical distribution of quadratic forms, constant of regression of quadratic statistic on linear form with random coefficient, reconstruction of a distribution through the common distribution of a set of linear forms with random coefficients. 

Let us give an example of a little bit nonstandard characteristic property of an exponential distribution. Namely, the relation (ref{eq8}) shows that the distribution of linear form $(1-p)X+\varepsilon_p Y$ does not depend on the parameter $p \in (0,1)$. It appears that this is a characteristic property of exponential distribution.

\begin{thm}\label{th7}
Let $X,Y$ be i.i.d. positive random variables. Suppose that $\{ \varepsilon_p, \; p\in (0,1)\}$ is a family of Bernoulli random variables, $\p\{\varepsilon_p =1\}=p$ and $\p\{\varepsilon_p =0\}=1-p$. The distribution of linear form $L=(1-p)X + \varepsilon_p Y$ does not depend on the parameter $p \in (0,1)$ if and only if $X$ has an exponential distribution.
\end{thm}
\begin{proof} If the distribution of linear form $L=(1-p)X + \varepsilon_p Y$ does not depend on the parameter $p \in (0,1)$ then its Laplace transform 
\begin{equation}\label{eq20}
f\bigl((1-p)t\bigr)\Bigl( 1-p+pf(t)\Bigr) = \psi (t)
\end{equation}
possesses the same property. The relation (\ref{eq20}) shows that the function $ff\bigl((1-p)t\bigr)$ is differentiable with respect to $p$ and, therefore, according to $t\geq 0$. The latest implies that $X$ possesses finite first moment. 

Let us differentiate the both sides of (\ref{eq20}) with respect to $p$:

\[ -t f^{\prime}\bigl((1-p)t \bigr) \Bigl( (1-p)+pf(t)\Bigr) - f\bigl((1-p)t \bigr) \bigl(1-f(t)\bigr). \]  
Passing here to limit as $p \to 1$ we obtain
\[ f(t) = \frac{1}{1+t \E X}. \]

\vspace{-0.2cm}
\end{proof}

In the book (\cite{KKM}, p.p. \hspace{-10
pt}153--157) there was given a characterization of Marshall-Olkin law by the property of identical distribution of a monomial and a linear form with random matrix coefficient. It would be interesting to study possibility of characterization of Marshall-Olkin distribution by independence of suitable statistics. 

Some of properties mentioned above will be a subject for our future work. 

\section*{Acknowledgment}

The study was partially supported by grant GA\v{C}R 19-04412S (Lev Klebanov).

\end{document}